\documentclass[12pt,a4paper]{amsart}
\usepackage[utf8]{inputenc}
\usepackage[T1]{fontenc}
\usepackage[UKenglish]{babel}
\usepackage[a4paper,margin=1in]{geometry}

\usepackage{amsmath,amsthm,amsfonts,amssymb}
\usepackage{mathrsfs,dsfont}
\usepackage{graphicx}
\usepackage{url}
\usepackage{verbatim}
\usepackage[dvipsnames]{xcolor}

\usepackage{import}
\usepackage{xifthen}
\usepackage{pdfpages}
\usepackage{transparent}

\usepackage{hyperref}
\usepackage{mathtools}

\usepackage{bbm}

\usepackage{marginnote}

\graphicspath{ {./figs/} }

\makeatletter
\def\namedlabel#1#2{\begingroup
	#2%
	\def\@currentlabel{#2}%
	\label{#1}\endgroup
}
\makeatother
\DeclarePairedDelimiter\norm{\lVert}{\rVert}%

\theoremstyle{plain}
\newtheorem{theorem}{Theorem}[section]
\newtheorem{corollary}[theorem]{Corollary}
\newtheorem{lemma}[theorem]{Lemma}

\newtheorem{proposition}[theorem]{Proposition}

\theoremstyle{definition}

\numberwithin{equation}{section}

\renewcommand\labelenumi{\textup{\alph{enumi})}}

\renewcommand\theenumi\labelenumi%{\textup{\alph{enumi})}}
\makeatletter\renewcommand{\p@enumii}{}\makeatother %\ref concatenates the \theenumi's with
\renewcommand{\leq}{\leqslant}

\renewcommand{\geq}{\geqslant}

\newcommand{\nI}{\textup{I}}
\newcommand{\R}{\mathds{R}}

\newcommand{\I}{\mathds{1}}

\newcommand{\pr}{\mathbf{P}}

\newcommand{\ex}{\mathbf{E}}

\begin{document}
	
	\title[Ground state for Schr\"odinger operator]
	{Ground state decay for Schr\"odinger operators with confining potentials}
	
	\date{\today}
	
	\author[M.~Baraniewicz]{Miłosz Baraniewicz}
	\address[M.~Baraniewicz]{Faculty of Pure and Applied Mathematics\\ Wroc{\l}aw University of Science and Technology\\ ul. Wybrze{\.z}e Wyspia{\'n}skiego 27, 50-370 Wroc{\l}aw, Poland}
	\email{milosz.baraniewicz@pwr.edu.pl}

		\begin{abstract}
		We give two-sided estimates of a ground state for Schr\"odinger operators with confining potentials. We propose a semigroup approach, based on resolvent and the Feynman--Kac formula, which leads to a new, rather short and direct proof. Our results take the sharpest form for slowly varying, radial and increasing potentials.
	\medskip

\noindent
\emph{Key-words}: eigenfunction, slowly varying potentials, Feynman--Kac formula, heat kernel, integral kernel, resolvent

\medskip

\noindent
2010 {\it MS Classification}: 47D08, 60J65, 35K05, 26A12\\
	\end{abstract}
	
	\footnotetext{ 
			Research was supported by National Science Centre, Poland, grant no. 2019/35/B/ST1/02421.
	}

	\maketitle
	
	\section{Introduction}
	The goal of the paper is to give the upper and the lower estimate at inifnity of the ground state eigenfunction for the Schr\"odinger operator 
	\[
	H = -\Delta+V, \quad \text{acting in} \ L^2(\R^d), \ d \geq 1,
	\]
	where $\Delta$ is the classical Laplacian and $V:\R^d \to [0,\infty)$ is a confining (i.e.\ $V(x) \to \infty$ as $|x| \to \infty$), locally bounded potential. 
	The spatial behaviour of eigenfunctions for Schrödinger operators involving the Laplacian and more general second order differential operators are now a classical topic, see Agmon \cite{Agmon,Agmon2}, Reed and Simon \cite{Reed-Simon},
and Simon \cite{Simon}.  There is a huge literature concerning the estimates of ground state eigenfunctions of $H$ with confining potentials. Here we refer the reader to the celebrated works by Carmona and Simon \cite{Carmona-Simon}, Carmona \cite{Carmona} and the other contributions discussed and quoted in these papers. The sharpest known result (see \cite[Section 4]{Davies-book}) was obtained for the power-type potential $V(x)= |x|^{2\beta}$, where $\beta>1$, and it says that the ground state eigenfunction $\varphi_0(x)$ is comparable at infinity to the function
\[
|x|^{-\beta/2+(d-1)/2} \exp\left(-\frac{1}{1+\beta} |x|^{1+\beta}\right). 
\]
Observe that $|x|^{1+\beta} = \sqrt{V(x)} |x|$. The best pointwise estimates that apply to more general continuous confining potentials were obtained in \cite{Carmona-Simon}. These bounds are given in terms of the Agmon distance and for sufficiently regular potentials they are sharp enough to ensure that
\begin{align} \label{eq:lim_gs}
\lim_{|x| \to \infty} \frac{-\log \varphi_0(x)}{\varrho(x)} = 1,
\end{align}
for an explicit function $\varrho$. We note in passing that when $\Delta$ is replaced with a non-local L\'evy operator $L$, then more is known -- see \cite{Kaleta-Lorinczi} for sharp two-sided estimates of $\varphi_0$ for a large class of $L$'s and $V$'s. 

The following theorem is the main result of this paper. 
	
	\begin{theorem}\label{th: ground state bound}
			Let $\varphi_0$ be the ground state of the Schr\"odinger operator $H=-\Delta+V$ with locally bounded, confining potential $V:\R^d \to [0,\infty)$. Then the following hold. 	
		\begin{itemize} 			
			\item[(I)] For every $\varepsilon >0$ and $\delta \in (0,1)$ there exist $c , r>0$ such that
			\begin{equation*}
				\varphi_0(x) \geq  c   \exp \left(-(1 + \varepsilon)\sqrt{ V^{\delta}(x)} |x| \right), \quad |x| \geq r,
			\end{equation*}
			where 
					\begin{equation} \label{eq:upper_prof}
				V^{\delta}(x):= \sup\limits_{z \in B_{|x|+\delta}(0)} V(z).
			\end{equation}
			\item[(II)] For every $\varepsilon, \delta \in (0,1)$ there exist $c, r>0$ such that 
			\begin{equation*}
				\varphi_0(x) \leq c  \exp\left(- (1-\varepsilon ) \delta \sqrt{ V_{\delta}(x)} \, |x|\right), \quad  |x| \geq r,
			\end{equation*}
			where
					\begin{equation} \label{eq:lower_prof}
				V_{\delta}(x) := \inf\limits_{z \in B_{\delta|x|}(x)} V(z).
			\end{equation}
		\end{itemize}
	\end{theorem}
	
	\noindent
Note that for every $\delta \in (0,1)$ we have $V^{\delta}(x), V_{\delta}(x) \to \infty$ as $|x| \to \infty$.

The proof of Theorem \ref{th: ground state bound} uses the semigroup-of-operators and resolvent techniques. It is inspired by the qualitatively sharp two-sided estimates of integral kernels for Schr\"odinger semigroups which we have obtained just recently in \cite{Baraniewicz-Kaleta}. We propose an argument which allows one to derive the lower and the upper estimate of the ground state from the respective estimate of the integral kernel of the semigroup $\big\{e^{-tH}: t \geq 0 \big\}$, without loosing too much information. Our approach leads to rather short and direct proofs. We believe this method can also be used in different settings, for different types of potentials.  
	
	Estimates in Theorem \ref{th: ground state bound} take the sharpest form for radial and increasing potentials which grow at infinity not too fast.   
		\begin{corollary}\label{cr: slowly varying}
		Let $V(x) = g(|x|)$ for some increasing function $g:[0, \infty) \to [0, \infty)$ such that $g(r) \to \infty$ as $r \to \infty$.
		Then the following hold.
		\begin{itemize}
			\item[(I)] If for every $\varepsilon >0$ there exist $\delta, t_0 >0$ satisfying
			\[
			g((1 + \delta) t) \leq (1 + \varepsilon) g(t), \quad t > t_0,
			\] 
			then there are $c,r > 0$ such that
			\begin{equation*}
				\varphi(x) \geq c   \exp \left(-(1 + \varepsilon)\sqrt{ V(x)} |x| \right), \quad |x| \geq r.
			\end{equation*}
			\item[(II)] If for every $\varepsilon, \delta \in (0,1)$ there exists $t_0>0$ satisfying 
			\[
			g((1 - \delta) t) \geq (1 - \varepsilon) g(t), \quad t > t_0,
			\]
			then there are $c, r > 0$ such that
			\begin{equation*}
				\varphi(x) \leq c   \exp \left(-(1 - \varepsilon)\sqrt{ V(x)} |x| \right), \quad |x| \geq r.
			\end{equation*}
		\end{itemize}
	If the assumptions in \textup{(I)} and \textup{(II)} are satisfied, then \eqref{eq:lim_gs} holds with $\varrho(x) = \sqrt{ V(x)} |x|$. 
		
	In particular, if $g$ is a function slowly varying at infinity, i.e.\ for every $\lambda>0$, 
	\[
		\lim\limits_{r \to \infty }\frac{g(\lambda r)}{g(r)} = 1,
	\]
	then \textup{(I)} and \textup{(II)} hold. Moreover, if $g$ grows at infinity at most linearly fast, e.g.
	\[
	g(r) = ar^{\alpha}+b, \quad \text{with} \ \alpha > 0 \ \text{and} \ a,b>0,
	\]
	then \textup{(I)} holds as well.
	\end{corollary}
We remark that our results for radial, slowly increasing and continuous potentials can also be derived from the estimates obtained in \cite{Carmona-Simon}. However, when the potential grows at infinity very slow, our bounds seem to be more accessible.

The properties of Schr\"odinger operators and the corresponding semigroups in both local and non-local settings are rather well understood for potentials which grow fast at infinity; this is related to the \emph{intrinsic ultracontractivity} which provides a lot of regularity \cite{Davies-Simon}. On the other hand, the theory for potentials slowly growing or decaying at infinity just started to shape up. This is now a very active area of research in probabilistic potential theory, see e.g.\ \cite{ Baraniewicz-Kaleta,Chen-Wang1, Chen-Wang2, Kaleta-Schilling} (see also recent contribution to the theory of classical Schr\"odinger semigroups with singular potentials \cite{Bogdan-Dziubanski-Szczypkowski,Cho-Song,Jakubowski-Szczypkowski}).  We believe that our work will turn out to be a valuable contribution to this field both in terms of results and methods.

	\subsection*{Acknowledgement} I want to thank Kamil Kaleta for all the help and discussions during preparation of this paper.

	\section{Lower bound}
	
	Recall that the Schr\"odinger semigroup $\big\{e^{-tH}: t \geq 0\big\}$ consists of integral operators, i.e.\ there exists a continuous function $(0, \infty)  \times \R^2 \ni (t,x,y) \mapsto u_t(x,y)$ such that
	\begin{equation*}
	U_tf(x) = e^{-tH} f(x) = \int_{\R^d} u_t(x,y) f(y) dy, \quad f \in L^2(\R^d,dx), \ \ t>0.
	\end{equation*}
	We have $u_t(x,y) = u_t(y,x)$ and $0 <  u_t(x,y) \leq g_t(x,y)$, for all $x,y \in \R^d$ and $t>0$, where 
	\[
	g_t(x,y) = g_t(y-x) = \frac{1}{(4\pi t)^{d/2}}\exp\left(-\frac{|y-x|^2}{4t}\right).
	\]
 is the standard Gauss--Weierstrass kernel.
	Let $\varphi_0$ be the \emph{ground state} of $H$ i.e.\ the  unique  eigenfunction corresponding to  eigenvalue $\lambda_0 := \inf \sigma (H)$ > 0. It is known that $\varphi_0$ is bounded, continuous and strictly positive on $\R^d$. Clearly, we have 
	\begin{equation}\label{eq: ground state}
 e^{-\lambda_0 t}  \varphi_0 (x) = U_t \varphi_0(x),  \quad x \in \R^d,  \ t>0.
	\end{equation}	
 For more information on Schr\"odinger semigroups we refer the reader to \cite{Demuth-Casteren, Reed-Simon, Simon}. 
	
The resolvent  or $\lambda$-potential operator corresponding to Laplacian $\Delta$ is defined by 
	\begin{equation}\label{eq:resolvent}
		R_{\lambda}u(x) := \int_0^{\infty} \int_{\R^d} u(y) g_t (y-x) e^{-\lambda t} \, dy \, dt , \quad \lambda >0, \ x \in \R^d,
	\end{equation}
	for every nonnegative or bounded Borel function $u$. It is a convolution operator with the kernel
	\begin{equation*}
		r_{\lambda}(y) = \int_{0}^{\infty} e^{- \lambda t} g_{t}(y) dt, \quad y \neq 0.
	\end{equation*} 
	We remark that the operator $\Delta$ is the infinitesimal generator of the process $(X_t)_{t \geq 0}$, where $X_t=B_{2t}$ and $(B_t)_{t\geq0}$ is the standard Brownian motion with the variance $\sqrt{t}$ (or, equivalently, with generator $(1/2)\Delta$). Clearly, the resolvent kernel 
$\widetilde r_{\lambda}(y)$ corresponding to the operator $(1/2)\Delta$ is connected to $r_{\lambda}(y)$ through the relation 
\[
\widetilde r_{\lambda}(y) = \int_{0}^{\infty} e^{- \lambda t} g_{t/2}(y) dt 
                            = 2 \int_{0}^{\infty} e^{- 2\lambda t} g_{t}(y) dt = 2 r_{2\lambda}(y).
\]
	
%	We start by stating asymptotically correct lower bound of $r_{\lambda}$
	It is known (see e.g.\ \cite[Example 7.14]{Schilling}) that
		\begin{equation*}
			 \widetilde r_{\lambda}(y)  = \frac{1}{\pi^{d/2}} \left(\frac{\sqrt{2\lambda}}{2|y|}\right)^{\tfrac{d}{2}-1} K_{\tfrac{d}{2}-1} \left(\sqrt{2\lambda}|y| \right)
		\end{equation*}
		where $K_v$ is the modified Bessel function of the second kind \cite[10.25]{NIST:DLMF}. By \cite[10.25.3]{NIST:DLMF}, we have
		\begin{equation*}
			 \Bigg|\sqrt{\frac{2r}{\pi}}K_v(r)e^r - 1 \Bigg| \xrightarrow{r \to \infty} 0,
		\end{equation*}
		which implies that for every $\varepsilon >0$ there exist $\rho, c >0$ such that for $|y| \geq \rho$  and $\lambda \geq 1$ we have 
		\begin{equation}\label{ieq: resolvent lower bound} 
		r_\lambda(y) \geq c e^{-(1 + \varepsilon) \sqrt{\lambda} |y|}.
		\end{equation}
	
	 The following proposition is the key observation in this section. It allows one to derive the lower bound for the ground state $\varphi_0(x)$ from a certain form of the lower estimate of the kernel $u_t(x,y)$. The proof uses the resolvent. 
	
	\begin{proposition}\label{prop: lower bound}
		Let $W: \R^d \to [0, \infty)$ be a function such that $W(x) \xrightarrow{|x| \to \infty} \infty$. Assume that there exist $c_1, \rho_1 > 0 $ such  that  for $|x| > \rho_1$, $y \in B_1(0)$  and $t>0$  we have
		\begin{equation}\label{ieq: lower bound assumption}
				u_t(x,y) \geq c_1 \exp\left(- W(x)t\right) g_{t}(x,y).
		\end{equation} 
		Then for every $\varepsilon >0$ there are $c_2= c_2(\varepsilon)>0$ and $\rho_2 = \rho_2(\varepsilon)>0$ such  that  for $|x| > \rho_2$ we have
		\begin{equation}\label{ieq: lower bound thesis}
				\varphi_0(x) \geq c_2  e^{-(1 + \varepsilon)\sqrt{ W(x)} |x|}.
		\end{equation} 
	\end{proposition}
	
	\begin{proof}

		Fix $\varepsilon >0$. Integrating on both sides of  \eqref{eq: ground state} with  $\int_0^{\infty}  (\ldots)  dt$  and using Tonelli's theorem yield
		\begin{equation*}
			\varphi_0(x) = \lambda_0 \int_0^{\infty}  U_t \varphi_0(x) \, dt \geq \lambda_0 \int_{B_1(0)} \int_0^{\infty} u_t(x,y) \varphi_0(y) \, dt \,dy, \quad x\in \R^d.
		\end{equation*}
Further, by applying \eqref{ieq: lower bound assumption}, we get for $|x| > \rho_1$
		\begin{align*}
			\varphi_0(x) & \geq c_1 \lambda_0 \int_{B_1(0)} \int_0^{\infty} \exp\left(- W(x)t\right) g_t(x,y) \varphi_0(y) \, dt \, dy \\
			             &  =  c_1 \lambda_0  \int_{B_1(0)} r_{ W(x)}(y-x)\varphi_0(y) \, dy.
		\end{align*} 
		Let $\varepsilon'>0$ be such  that  $(1+\varepsilon')^2 = (1+ \varepsilon)$. We choose $ \rho_2>\rho_1$ in such a way that $|x-y| \leq (1 + \varepsilon') |x|$  and $W(x) \geq 1$ when $|x|>\rho_2$ and $y \in B_1(0)$,  and \eqref{ieq: resolvent lower bound} holds with $\varepsilon'$  and some constant $c>0$ when $|y|>\rho_2$. %and $y_0 = (1+\varepsilon') = y_0$. 
		
		Then, by  monotonicity of $r_{\lambda}(\cdot)$ and \eqref{ieq: resolvent lower bound}, we get for $|x|>\rho_2$ and $y \in B_1(0)$ 
		\begin{align*}
			 r_{W(x)}(x-y) & \geq r_{W(x)}\big((1 + \varepsilon') x\big) \\
			               & \geq c e^{-(1 + \varepsilon')^2 \sqrt{W(x)} |x|} = c e^{-(1 + \varepsilon) \sqrt{W(x)} |x|}.
		\end{align*}
		Consequently, 
		\begin{equation*}
			\varphi_0(x) \geq c_2  e^{-(1 + \varepsilon)\sqrt{W(x)} |x|}, \quad |x|>\rho_2,
		\end{equation*}	
		with $c_2 =  c c_1 \lambda_0 \int_{B_1(0)} \varphi_0(y) \, dy.$ 
	\end{proof}

We note that the estimate \eqref{ieq: lower bound assumption} is inspired by the lower bound in \cite[Lemma 5.1]{Baraniewicz-Kaleta}, but the original statement is not sharp enough for our purposes. We therefore revisit its proof to make it as sharp as needed here. Recall that $V^{\delta}$ is defined in \eqref{eq:upper_prof}.

	\begin{lemma}\label{lem: lower bound assumption}
		For every $\varepsilon > 0$ and $\delta >0$ there exist $ c_1, \rho_1 > 0$ such that  for $|x| > \rho_1$, $y \in B_1(0)$ and $t>0$ we have
		\begin{equation*}
			u_t(x,y) \geq  c_1  \exp\left(- (1 + \varepsilon)V^{\delta}(x)t\right) g_{t}(x,y).
		\end{equation*}
	\end{lemma}
	
	\begin{proof}
		 Fix $\varepsilon >0$ and $\delta >0$.  We choose $\rho_1 >2 $ in such a way that for $ |x| > \rho_1$ we have
		\begin{equation}\label{ieq: how big x}
			\frac{\mu_0 }{(|x|+ \delta)^2} + \frac{1}{\delta^2} \leq \varepsilon V^{\delta}(x),
		\end{equation}
		 where $\mu_0 > 0$ is the principal eigenvalue of the operator $-\Delta_{B_1(0)}$ -- the positive Dirichlet Laplacian on the unit ball $B_1(0)$. 
		
		Recall that the operator $\Delta_{B_r(0)}$ is the infinitesimal generator of the process $(X_t)_{t \geq 0}$ killed upon exiting an open ball $B_r(0)$. The transision densities of this process are denoted by $g_t^{B_r(0)}(x,y)$, i.e.\ we have
		\[
		\pr_x\big(X_t \in A, t < \tau_{B_r(0)}\big) = \int_A g_t^{B_r(0)}(x,y) dy, \quad x \in B_r(0), \ t >0,
		\]
		for every Borel subset $A$ of $B_r(0)$. Here by $\tau_{B_r(0)}$ we denote the first exit time of the process $(X_t)_{t \geq 0}$ 
		from $B_r(0)$, and $\pr_x$ is the probability of the process starting at $x \in B_r(0)$. 
		From \cite[Corollary 1]{malecki-serafin} we know that there exists a constant $ c \in (0,1]$ such that
		\begin{align} \label{eq:lower_dir_balls}
			g^{B_r(0)}_t(x,y) \geq c \frac{1 \wedge \frac{(r-|x|) (r-|y|)}{t}}{(1 \wedge \frac{r^2}{t})^{(d+2)/2}} \exp\left(- \mu_0 \frac{t}{r^2}\right) g_{t}(x,y), \quad r>0, \ x,y \in B_r(0),\,  t>0.
		\end{align}
		
		By the representation of the kernel $u_t$ from \cite[Proposition 2.7]{Demuth-Casteren}, for $|x| >2, y \in B_{1}(0)$ we have
		\begin{align*}
			u_t(x,y) &\geq   \lim\limits_{s \nearrow t} \ex_x \left[ e^{-\int_0^s V(X_u)du} g_{t-s}(X_s,y)\ ; \ s < \tau_{B_{|x|+\delta}(0)} \right] \\
			& \geq e^{-t V^{\delta}(x)}  \lim\limits_{s \nearrow t} \ex_x \left[ g^{B_{|x|+\delta}(0)}_{t-s}(X_s,y)\ ; \ s < \tau_{B_{|x|+\delta}(0)} \right] \\
			& = e^{-t V^{\delta}(x)}  g^{B_{|x|+\delta}(0)}_t(x,y).
		\end{align*} 
	  Then, by using  \eqref{eq:lower_dir_balls}, the inequality $1 \wedge a \geq e^{-1/a}$ valid for $a>0$, and \eqref{ieq: how big x}, we obtain
		\begin{align*}\label{eq:lem-aux-1-1}
			g^{B_{|x|+\delta}(0)}_t(x,y) &\geq  c \Big(1 \wedge \frac{\delta^2}{t} \Big)  \exp\left(- \mu_0 \frac{t}{(|x|+\delta)^2}\right) g_t(x,y) \\ 
			&\geq  c \exp\left(- \Big(\frac{\mu_0 }{(|x|+\delta)^2} + \frac{1}{\delta^2}\Big)t\right) g_{t}(x,y) \\
			&\geq  c \exp\left(- \varepsilon V^{\delta}(x) t\right) g_{t}(x,y) \quad \text{ for } \quad |x| > \rho_1, \ y\in B_1(0),
		\end{align*}
		and, consequently,
		\begin{equation*}
			u_t(x,y) \geq  c \exp\left(- (1 + \varepsilon) V^{\delta}(x) t\right) g_{t}(x,y), \quad |x| > \rho_1, \ y \in B_1(0).
		\end{equation*}
	\end{proof}
	
	\section{Upper bound}
	
	 The following result is a counterpart of Proposition \ref{prop: lower bound} for the upper bound.
	
	\begin{proposition}\label{prop: upper bound}
	Let $W: \R^d \to (0, \infty)$ be a function such that $W(x) \xrightarrow{|x| \to \infty} \infty$.
		Assume that there are constants $ c_1,A >0$ and $a { \geq } 1$ such that for all $x,y \in \R^d$,  $x \neq 0$, and $t>0$ we have 
		\begin{equation}\label{ieq: upper bound assumption}
				u_t(x,y) \leq c_1  \exp\left(-\left(A W(x)t \, \wedge \,  \sqrt{W(x)} \, |x|\right)\right) g_{at} (x,y).
		\end{equation} 

		Then for  every  $\varepsilon >0$ there are $ c_2, \rho>0$ such  that  for $|x|>\rho$ we have
		\begin{equation}\label{ieq: uper bound thesis}
			\varphi_0(x) \leq c_2 \norm{\varphi_0}_{\infty}  \exp\left(- (1-\varepsilon ) \sqrt{W(x)} \, |x|\right).
		\end{equation} 
	\end{proposition}
	
	\begin{proof}
		Fix $\varepsilon >0$ and set $ \rho >0$ such that $|x|>\rho$ implies $\frac{\lambda_0}{\varepsilon A} \leq W(x)$. Let $|x| > \rho$.  By \eqref{eq: ground state} and \eqref{ieq: upper bound assumption} we have 

		\begin{equation*}
			\varphi_0(x) \leq  \norm{\varphi_0}_{\infty} e^{\lambda_0 t}  U_{t} \I(x) \leq c_1 \norm{\varphi_0}_{\infty} e^{\lambda_0 t}  \exp\left(-\left(A W(x)t \, \wedge \,  \sqrt{W(x)} \, |x|\right)\right),
		\end{equation*}
		 for all $t>0$.  Choosing $t = \frac{ |x|}{A \sqrt{W(x)}}$ yields
		\begin{equation*}
			\varphi_0(x) \leq c_1\norm{\varphi_0}_{\infty} e^{\lambda_0 \frac{ |x|}{A \sqrt{W(x)}}}  \exp\left(-  \sqrt{W(x)} \, |x|\right).
		\end{equation*}
		Because  $ |x| > \rho$, we finally get 
		\begin{equation*}
			\varphi_0(x) \leq c_1 \norm{\varphi_0}_{\infty}  \exp\left(- (1-\varepsilon ) \sqrt{W(x)} \, |x|\right).
		\end{equation*}	
	\end{proof}

The structure of \eqref{ieq: upper bound assumption} is motivated by \cite[Lemma 4.5]{Baraniewicz-Kaleta}. However, our application in this paper require a stronger version of this bound and, therefore, we need to improve it. The proof differs from the original one in some details. We decided to present here an almost complete reasoning for the reader convenience. Recall that $V_{\delta}$ is defined in \eqref{eq:lower_prof}.  

	\begin{lemma}\label{lem : upper bound assumption}
		For every $\varepsilon, \delta \in (0,1)$  there are constants $a, c>1$ such that for all $x,y \in \R^d$, $x \neq 0$,  $t>0$ we have
		\begin{equation}\label{ieq: upper bound lemma}
			u_t(x,y) \leq c  \exp\left(-\left(\frac{1}{a}V_{\delta}(x)t \, \wedge \,  (1-\varepsilon)\delta\sqrt{V_{\delta}(x)} \, |x|\right)\right) g_{at} (x,y).
		\end{equation} 
	\end{lemma}
	
	\begin{proof}
	 Fix $\varepsilon>0$ and $\delta \in (0,1)$. Let $x,y \in \R^d$,  $x \neq 0$,  $t>0$ and let $b>1$ be such $(1-\varepsilon) \sqrt{b} < 1$. %If $x=0$, then $e^{-C_4\sqrt{V_*(x)}|x|}$=1. Consequently, 

Denote $a=b/(b-1)$ so that $1/a+1/b=1$. 
	We have
	\begin{align*}
		u_t(x,y) = \int_{\R^d} u_{t/a}(x,z) u_{t/b}(z,y) dz & = \ex_x \left[e^{-\int_{0}^{t/a}V(X_s)ds} u_{t/b} (X_{t/a},y)\right] \\
		& = \ex_x \left[e^{-\int_{0}^{t/a}V(X_s)ds} u_{t/b} (X_{t/a},y) ; \, t/a < \tau_{B_{{\delta|x|}}(x)}\right] \\
		& \ \ \ \ \ \ \ \ + \ex_x \left[e^{-\int_{0}^{t/a}V(X_s)ds} u_{t/b} (X_{t/a},y) ; \, t/a \geq \tau_{B_{\delta|x|}(x)}\right] \\
		& =: \nI_1 + \nI_2,
	\end{align*}
	and further,
	\[
	\nI_1 \leq e^{-(t/a)V_{\delta}(x)} \ex_x \left[ g_{t/b}(X_{t/a},y);  \, t/a < \tau_{B_{{\delta|x|}}(x)}\right] ,
	\]
	\[
	\nI_2 \leq \ex_x \left[e^{-\int_0^{\tau_{B_{\delta|x|}(x)}}V(X_s)ds}g_{t/b}(X_{t/a},y)\right] \leq \ex_x \left[e^{- V_{\delta}(x)\tau_{B_{\delta|x|}(x)}}g_{t/b}(X_{t/a},y) \right].
	\]
By Hölder's inequality, we get 
	\begin{align}\label{eq:integrals1}
		\nI_1 & \leq  e^{-(t/a)V_{\delta}(x)} \pr_x\big(t/a < \tau_{B_{{\delta|x|}}(x)}\big)^{\frac{1}{b}} \ex_x \left[(g_{t/b}(X_{t/a},y))^{a}  \right]^{\frac{1}{a}} \\ & \leq e^{-(t/a)V_{\delta}(x)} \ex_x \left[(g_{t/b}(X_{t/a},y))^{a}  \right]^{\frac{1}{a}}, \nonumber
	\end{align} 
	\begin{equation}\label{eq:integrals2}
		\nI_2 \leq	\ex_x \left[e^{-bV_{\delta}(x) \tau_{B_{\delta|x|}(x)}} \right]^{\frac{1}{b}}
		\ex_x \left[(g_{t/b}(X_{t/a},y))^{a}  \right]^{\frac{1}{a}}.
	\end{equation}
 Estimate in \cite[ Remark 4.4]{Baraniewicz-Kaleta}  applied to $\lambda = bV_{\delta}(x)$ and $r = \delta|x|$ gives
	\begin{align*}
	\ex_x \left[e^{-bV_{\delta}(x)\tau_{B_{\delta|x|}(x)}} \right]^{\frac{1}{b}}& = \ex_0 \left[e^{-b V_{\delta}(x)\tau_{B_{\delta|x|}(0)}} \right]^{\frac{1}{b}} \\
	&\leq C^{1/b} e^{-(1 - \varepsilon)\sqrt{b} \delta \frac{   \sqrt{V_{\delta}(x)} \, |x|}{\sqrt{b}}} = C^{1/b} e^{-(1 - \varepsilon) \delta \sqrt{V_{\delta}(x)} \, |x|},
\end{align*}
	 with the constant $C=C(\varepsilon) \geq 1$, and 
	\begin{align*}
		\ex_x \left[(g_{t/b}(X_{t/a},y))^{a}  \right]^{\frac{1}{a}} \leq a^{\frac{d}{2}}b^{\frac{d}{2b}} g_{at}(x,y) 
	\end{align*}
	as in the original proof. 
	Using these bounds we can continue estimating in \eqref{eq:integrals1}, \eqref{eq:integrals2}:\ 
	\[
	\nI_1 \leq  C^{1/b} a^{\frac{d}{2}}b^{\frac{d}{2b}}  \exp\left( -\frac{t}{a}V_{\delta}(x)\right)g_{at}(x,y),
	\]
	\[
	\nI_2 \leq C^{1/b} a^{\frac{d}{2}} b^{\frac{d}{2b}} \exp\left(-(1-\varepsilon)\delta\sqrt{V_{\delta}(x)} \, |x|\right) g_{at}(x,y),
	\] 
which immediately leads to a conclusion. 
\end{proof}

	\section{The proof of the main theorem}

	\begin{proof}[Proof of Theorem \ref{th: ground state bound}]
		For the proof of  the lower bound we set  $\varepsilon >0$ and $\delta \in (0,1)$. We fix $\varepsilon' >0$ such that $(1 + \varepsilon')^{3/2} = 1 + \varepsilon$ and use the estimate in Lemma \ref{lem: lower bound assumption} for $\varepsilon'$ as assumption \eqref{ieq: lower bound assumption} with $W(x) = (1+ \varepsilon')V^{\delta}(x)$ in Proposition \ref{prop: lower bound}. The final estimate is then obtained as \eqref{ieq: lower bound thesis} for the same $\varepsilon'$.  
		
		 The upper bound is proved similarly, we fix $\varepsilon, \delta \in(0,1)$. We set $ \varepsilon' >0$ in such a way that $(1-\varepsilon')^2 = (1-\varepsilon)$ and use the estimate in Lemma \ref{lem : upper bound assumption} with such $\varepsilon'$  as assumption \eqref{ieq: upper bound assumption} with $ W(x) = (1+\varepsilon')^2\delta^2 V_{\delta}(x)$ and $ A = 1/(a(1-\varepsilon')^2\delta^2)$, where $a$ comes from Lemma \ref{lem : upper bound assumption}, in Proposition \ref{prop: upper bound}. Again, the claimed bound is obtained as the estimate \eqref{ieq: uper bound thesis} for $\varepsilon'$.
	\end{proof}

	\begin{proof}[Proof of Corollary \ref{cr: slowly varying}]
		It is enough to observe that  by respective assumptions in Parts (I) and (II) we have the following:\ for every  $\varepsilon>0$ there exists $ \rho \geq 1$  and $\delta \in (0,1)$ such that  for $|x| > \rho$ we have
		\begin{equation*}
			V^{\delta}(x) = g(|x| + \delta ) \leq g((1+\delta)|x|) \leq (1 + \varepsilon) g(|x|) = (1 + \varepsilon) V(x)
		\end{equation*}
		and,  for every $\varepsilon ,\delta \in (0,1)$ there exists $ \rho \geq 1$ such that for $|x| > \rho$ we have 
		\begin{equation*}
			V_{\delta}(x) = g((1 - \delta) |x|) \geq (1 - \varepsilon) g(|x|) = (1 - \varepsilon) V(x).
		\end{equation*}
		Then the  assertion in (I)  is a straightforward consequence of the lower estimate in Theorem \ref{th: ground state bound}, and the assertion in (II) follows from the the upper bound in Theorem \ref{th: ground state bound} because we can choose $\delta$ to be arbitrarily close to $1$.
	\end{proof}

\end{document}